%% file: main.tex
\documentclass[11pt]{article}
\usepackage[utf8]{inputenc}
\usepackage[T1]{fontenc}
\usepackage{amsmath,amssymb,amsthm}
\usepackage[margin=1.2in]{geometry}
\usepackage{microtype}
\usepackage{url}
\usepackage{booktabs}
\usepackage{array}
\usepackage[hidelinks]{hyperref}

\theoremstyle{plain}
\newtheorem{theorem}{Theorem}[section]
\newtheorem{lemma}[theorem]{Lemma}

\theoremstyle{definition}

\newcommand{\N}{\mathbb{N}}

\newcommand{\Coal}{\sim}

\title{Kernel-Checked Frontier Certificates\\ for Erd\H{o}s Problem 414}
\author{Ibrahim Mian\thanks{Corresponding author. Email: \texttt{ibrahimnmian@gmail.com}.} \qquad Shayaan Siddique\\[2pt]
\normalsize Millennium Research}
\date{}

\begin{document}
\maketitle

\begin{abstract}
Let $h(n)=n+\tau(n)$, where $\tau$ counts divisors. Erd\H{o}s and Graham asked, after Spiro, whether the orbits of any two positive integers under $h$ eventually share a point (Problem~414 on Bloom's list). The problem is open, and a finite computation cannot close it; nothing about it has been checked by a proof kernel, and the statement in the \texttt{formal-conjectures} repository carries a \texttt{sorry}. We replay what can be certified through the Lean~4 kernel under an axiom gate (axioms exactly \texttt{propext}, \texttt{Classical.choice}, \texttt{Quot.sound}; no \texttt{sorry}; no \texttt{native\_decide}). Three elementary facts recorded by Li---coalescence is an equivalence relation, $\tau(n)$ is odd exactly for squares, and $\tau(n)$ is bounded by $2\sqrt n$ so that orbits skip no square annulus $[k^2,(k+1)^2)$---are formalized, and his frontier lemma is formalized in a window form: every orbit started below a level $N$ has a point within $\lceil 2\sqrt N\rceil$ below $N$. The frontier lemma becomes a certificate: the $\lceil 2\sqrt N\rceil$ values $\tau(m)$ just below $N$ and the orbit points above $N$ until the orbits through the crossing points have merged. Its soundness is a kernel-checked theorem, and the certified rungs state that every pair of positive integers below $N$ coalesces for $N=10^5,10^6,10^7,10^8$; the certificate at $10^8$ has $44{,}530$ entries rather than $10^8$. Every ledger and every quoted count was produced by two programs sharing no code that agree by hash. We record what is known about a non-coalescing pair and measure the exit-set sizes Li bounds. Nothing here is a proof of the conjecture.
\end{abstract}

\section{The problem, and what this paper does not do}
\label{sec:intro}

Problem~414 on Bloom's list \cite{EP414} asks, following Spiro (see \cite[p.~82]{ErGr80}): with $h(n)=n+\tau(n)$ and $h^{[k]}$ its $k$-th iterate, is it true that for any $m,n$ there are $i,j$ with $h^{[i]}(m)=h^{[j]}(n)$? (Li \cite{Li26} writes $T$ for the map.) Erd\H{o}s and Graham believed the answer is yes. The orbit of $1$ is OEIS A064491 \cite{OEIS}, whose entry cites Spiro's 1977 West Coast Number Theory problem and, in a 2013 note by Sloane, poses two questions: whether every starting value joins this sequence, and whether the parity of its terms changes infinitely often. The problem is open. A disproof is not a finite object (two orbits that never meet), so no computation settles it, and nothing below claims to.

What is certified is smaller. Li \cite{Li26} reformulated the problem through the graph with edges $n \leftrightarrow h(n)$ and recorded, among analytic estimates we do not reproduce, three elementary facts and a frontier lemma. We formalize the facts (Section~\ref{sec:elementary}) and the lemma in a window form, turn it into a checkable certificate with a kernel-checked soundness theorem (Section~\ref{sec:cert}), and certify that every pair of positive integers below $10^8$ coalesces (Section~\ref{sec:rungs}). Section~\ref{sec:checklist} records what is known about a counterexample; Section~\ref{sec:measure} reports measurements; Section~\ref{sec:remains} says what remains.

Two tiers are kept apart throughout. \emph{Kernel-checked} means a Lean~4 theorem in the accompanying repository whose axiom closure lies within the set of \texttt{propext}, \texttt{Classical.choice} and \texttt{Quot.sound} (for every theorem named in this paper it is exactly that set), checked by a mechanical audit of every theorem in every module; the build and audit logs are in the repository. \emph{Cited} means a published statement we use as stated. Li's numbering is always prefixed by his name; unprefixed numbers are ours.

\section{Elementary facts, kernel-checked}
\label{sec:elementary}

Write $a\Coal b$ if $h^{[i]}(a)=h^{[j]}(b)$ for some $i,j\ge 0$ (Lean: \texttt{Coal}). The definition of $h$ is taken verbatim from the statement of the problem in the \texttt{formal-conjectures} repository \cite{FC}, so the rung theorems below are about the object that statement is about. Table~\ref{tab:leanmap} maps every statement of \cite{Li26} we formalize to its Lean name. The Lean names in the table are generated from the sources by a script and checked to exist; Li's numbering is entered by hand and was checked against his text.

\begin{table}[h]
\centering
\input{leanmap}
\caption{Statements of \cite{Li26} and the Lean theorems establishing the forms stated here (one-step parity, integer square-root bound, window form of the frontier lemma). The last two rows are ours.}
\label{tab:leanmap}
\end{table}

\begin{lemma}[Li \cite{Li26}, Lemma 2.1]
\label{lem:equiv}
$\Coal$ is an equivalence relation on $\N$.
\end{lemma}

\begin{proof}
Reflexivity and symmetry are immediate. For transitivity, from $h^{[i]}(a)=h^{[j]}(b)$ and $h^{[k]}(b)=h^{[l]}(c)$ one has $h^{[k+i]}(a)=h^{[k]}(h^{[j]}(b))=h^{[j+k]}(b)=h^{[j]}(h^{[l]}(c))=h^{[j+l]}(c)$; in Lean this is four rewrites with \texttt{Function.iterate\_add\_apply} and one with \texttt{add\_comm} (\texttt{coal\_trans}).
\end{proof}

\begin{lemma}[classical; Li \cite{Li26}, \S 6]
\label{lem:parity}
For $n\ge 1$, $\tau(n)$ is odd if and only if $n$ is a square. Hence $h(n)\equiv n \pmod 2$ unless $n$ is a square, and $h(n)\not\equiv n \pmod 2$ if it is.
\end{lemma}

\begin{proof}
Split the divisors of $n$ into those $d$ with $d^2<n$, those with $d^2>n$, and the at most one $d$ with $d^2=n$. The map $d\mapsto n/d$ is a bijection between the first two classes (\texttt{Finset.card\_nbij'} with \texttt{Nat.div\_div\_self} as the two-sided inverse), so $\tau(n)$ is twice the size of the first class plus the size of the third, which is $1$ when $n$ is a square and $0$ otherwise (\texttt{odd\_card\_divisors\_iff}). The two parity statements follow by reducing $n+\tau(n)$ modulo $2$ (\texttt{h\_eq\_of\_parity}, \texttt{h\_ne\_of\_parity\_square}). We found no Mathlib lemma for the parity of $\tau$. Consequently two orbits of opposite parity can share a point only after one of them has passed through a square; that iterated statement is Li's Proposition~6.2, which we formalize only in the one-step form above.
\end{proof}

\begin{lemma}[Li \cite{Li26}, Lemma 4.1, integer form]
\label{lem:sqrt}
$\tau(n)\le 2\lfloor\sqrt n\rfloor$, and $h(n) < (\lfloor\sqrt n\rfloor+2)^2$.
\end{lemma}

\begin{proof}
With the same split, the first class lies in $[1,\lfloor\sqrt n\rfloor]$, and in $[1,\lfloor\sqrt n\rfloor-1]$ when $n$ is a square, since $d^2<n$; the bound follows in both cases (\texttt{card\_divisors\_le\_two\_mul\_sqrt}, via \texttt{Nat.le\_sqrt} and \texttt{Nat.sqrt\_eq}). Then $h(n)\le n+2\lfloor\sqrt n\rfloor<(\lfloor\sqrt n\rfloor+1)^2+2\lfloor\sqrt n\rfloor<(\lfloor\sqrt n\rfloor+2)^2$, using $n<(\lfloor\sqrt n\rfloor+1)^2$ (\texttt{Nat.lt\_succ\_sqrt}); the arithmetic is discharged by \texttt{nlinarith} (\texttt{h\_lt\_sqrt\_add\_two\_sq}). Li states $h(n)<(\sqrt n+1)^2$ with the real square root; in the integer form the $+2$ is needed, since $h(12)=18\ge(3+1)^2$.
\end{proof}

The last lemma gives the window bound used below: if $4N\le W^2$ and $m<N$ then $\tau(m)<W$ (\texttt{card\_divisors\_lt\_of\_lt}), because $(2\lfloor\sqrt m\rfloor)^2\le 4m<4N\le W^2$.

\section{The frontier certificate}
\label{sec:cert}

Fix $N$ and $W$ with $4N\le W^2$ and $W<N$. Every positive $a<N$ has an orbit point in the window $[N-W,N)$: if $a<N-W$ then $h(a)=a+\tau(a)<N-W+W=N$ by the window bound, and $h(a)>a$, so descending induction on $N-a$ reaches the window (\texttt{exists\_iterate\_mem\_window}). Hence if every window element coalesces with one integer $t$, so does every positive $a<N$ (\texttt{coal\_of\_window}). For $m<N$ with $h(m)\ge N$ call $h(m)$ a crossing point of $N$; Li's crossing set $F_N$ is the set of crossing points, and the last orbit point below $N$ of any start lies in the window, so its crossing point is the successor of a window element.

A certificate lists $(m,\tau(m))$ for every $m$ in the window in increasing order, then $(x,\tau(x))$ for orbit points $x\ge N$ in increasing order, and a target $t$. The checker \texttt{certOk} replays it in two phases. In the window phase, each listed $m$ must be the expected next integer; its successor $m+\tau(m)$ is inserted into the pending list if it is at least $N$ (the checker's insertion keeps the list sorted and duplicate-free; the soundness proof uses only membership), and ignored otherwise (it is then a later window element). In the above phase, each listed $x$ must equal the head of the pending list; it is removed and $x+\tau(x)$ inserted. The replay passes if the window has exactly $W$ entries and exactly $[t]$ remains.

\begin{theorem}[\texttt{coal\_below\_of\_cert}]
\label{thm:cert}
Let $4N\le W^2$ and $W<N$. If the listed $\tau$ values, for the window and for the nodes above, are correct and the replay passes, then $a\Coal b$ for all $0<a,b<N$.
\end{theorem}

\begin{proof}
Two invariants. Above phase (\texttt{abovePend\_sound}): if every integer in the final pending list coalesces with $t$, so does every integer ever in it, by induction along the listed nodes, since a removed $x$ satisfies $h(x)=x+\tau(x)$, which lies in the next pending list. Window phase (\texttt{winPend\_sound}): by induction from the last window element downward, each $m$ either has $h(m)\ge N$, in which case $h(m)$ was inserted and coalesces with $t$ by the above invariant, or has $h(m)<N$, in which case $h(m)$ is a later window element already handled; here $\tau(m)\ge 1$ needs $m\ge1$, which is where $W<N$ enters. The window lemma then covers all positive $a<N$, and $a\Coal t\Coal b$ gives $a\Coal b$ by Lemma~\ref{lem:equiv}.
\end{proof}

Each $\tau$ value is established from a factorization certificate: a list of pairs $(p,e)$ with strictly decreasing prime heads and product $n$ gives $\tau(n)=\prod(e+1)$ (\texttt{tau\_of\_pairs}, by \texttt{Nat.Coprime.card\_divisors\_mul} and \texttt{Nat.divisors\_prime\_pow}); the primality of each head is re-proved by \texttt{norm\_num}, the product and the exponent count by \texttt{rfl}, and the replay itself is evaluated by the kernel (\texttt{decide +kernel}). Nothing in the certificate is trusted: a wrong $\tau$ fails the per-node theorem, and a wrong replay fails the kernel evaluation.

This is Li's Lemma~3.1 and Corollary~3.2 with the crossing set $F_N$ replaced by a window that needs no prior knowledge of which starts cross $N$; the window is $\lceil 2\sqrt N\rceil$ wide by Lemma~\ref{lem:sqrt}.

\section{Certified rungs}
\label{sec:rungs}

\begin{theorem}[\texttt{coal\_below\_1e5}, \dots, \texttt{coal\_below\_1e8}]
\label{thm:rungs}
For $N\in\{10^5,10^6,10^7,10^8\}$ and all $0<a,b<N$, $a\Coal b$.
\end{theorem}

\begin{table}[h]
\centering
\begin{tabular}{@{}rrrrr@{}}
\toprule
$N$ & $W$ & nodes above $N$ & $|F_N|$ & target \\
\midrule
$10^5$ & 633 & 903 & 11 & 104494 \\
$10^6$ & 2000 & 410 & 11 & 1002242 \\
$10^7$ & 6325 & 2844 & 12 & 10018018 \\
$10^8$ & 20000 & 24530 & 13 & 100180126 \\
\bottomrule
\end{tabular}
\caption{Certificate sizes: the window width $W=\lceil 2\sqrt N\rceil$, the number of orbit points $x\ge N$ listed, the number $|F_N|$ of crossing points of $N$, and the target, the integer that remains. The pending list never held more than 13 integers.}
\label{tab:rungs}
\end{table}

A certificate at level $N$ has $\lceil 2\sqrt N\rceil$ window entries plus the merge depth above $N$, so the certified frontier is not tied to computing every orbit below $N$: the $10^8$ ledger was produced by trial division on $44{,}530$ integers in seconds, by a C and a Python program sharing no code (the certificate-only pair). The same pair produces a $10^9$ ledger ($W=63{,}246$, $252{,}486$ nodes above, $20$ crossing points) in a minute; that ledger is distributed and verified by trial division, and it is not certified, since its Lean build is projected at about a day of one core. Each of the four certified ledgers was produced by two programs sharing no code (a C sieve and a numpy pointer-doubling implementation for $N\le10^7$; the certificate-only pair for every level) that agree by SHA-256, and re-checked by a third program before being laid out as Lean.

\section{What is known about a counterexample}
\label{sec:checklist}

Suppose $a\not\Coal b$ for positive $a,b$.
\begin{enumerate}
\item At least one of $a$, $b$ lies in a component containing no positive integer below $10^8$ (Theorem~\ref{thm:rungs}, kernel-checked).
\item If $a$ and $b$ have opposite parity, then a square lies on one of the two orbits; equivalently, a pair of opposite parity with no square on either orbit does not coalesce, and a coalescing pair of opposite parity meets only after a square (one-step form kernel-checked, \texttt{h\_eq\_of\_parity}; the iterated form is Li's Proposition~6.2, cited).
\item The number of components meeting $[1,X]$ is at most $\log X+2\gamma+O(X^{-1/4})$, so the least elements of distinct components are exponentially sparse; this constrains the set of components, not the pair (Li, Theorem~3.4 and Corollary~3.7; cited).
\item Each of $a,b$ whose orbit misses the orbit of $1$ (at least one does) determines an infinite path in Li's thin annular graph that avoids the main branch at every sufficiently large level (Li, Proposition~5.3; cited).
\end{enumerate}

\section{Measurements}
\label{sec:measure}

Li's exit sets $E_k=\{\tau(k^2-j)-j : 1\le j\le 2k-2,\ \tau(k^2-j)\ge j\}$, the active-deficit counts $a_k=\#\{1\le j\le 2k-2 : \tau(k^2-j)\ge j\}$ (so $|E_k|\le a_k$), his annular transfer map $A_k$ (for an offset $r$, $A_k(r)=x-(k+1)^2$ where $x$ is the first point of the orbit of $k^2+r$ at or beyond $(k+1)^2$; Li, Definition~4.2), the one-step widths $|A_k(E_k)|$, and the confluence sets $W_{k,0}=E_k$, $W_{k,s}=A_{k+s-1}(W_{k,s-1})$ (Li, Definition~5.7; not to be confused with the window width $W$) were computed for $2\le k\le 5000$ by two programs sharing no code, agreeing by hash. Li's Corollary~5.10 says that $\liminf_k|A_k(E_k)|=1$ would imply the conjecture, and his Remark~8.15 proposes $K\log K$ as the scale of $\sum_{k\le K}a_k$. Measured, the ratio $\sum_{k\le K}a_k/(K\log K)$ moves from $1.69$ at $K=10$ to $1.90$ at $K=5000$, and $\sum_{k\le K}|E_k|/(K\log K)$ from $1.30$ to $1.56$. One-step collapse $|A_k(E_k)|=1$ occurs for $k=2,3,5,315$ and no other $k\le 5000$. For every $k\le 5000$ some $s$ has $|W_{k,s}|=1$; the least such $s$ has mean $10.7$ and median $8$. These are measurements, not theorems.

\section{What remains}
\label{sec:remains}

The conjecture is as open as before. The rungs certify a finite statement. Li's finite-level criteria (his Theorem~4.4 and Corollaries~5.5 and~5.10) are equivalences or sufficient conditions he proves; none of them has been shown to hold, and this paper does not attempt any of them.

\section*{Acknowledgment}
The elementary statements formalized in Section~\ref{sec:elementary} and the frontier lemma behind Section~\ref{sec:cert} are recorded in Li's paper \cite{Li26}; his numbering is given beside each Lean theorem. This work was carried out without correspondence with him.

\section*{Code and data}
% Zenodo DOI: to be added on upload day (docs/ARXIV_SUBMISSION.md, step 2)
The Lean sources, scripts, ledgers and the recorded build logs are in the repository \emph{hastatus}, release \texttt{v1}, compiled by Lean~4 (v4.30.0) against Mathlib (revision \texttt{c5ea003}) under the axiom gate of Section~\ref{sec:intro}; the tag fixes the commit. The repository is at
\begin{center}\url{https://github.com/ibrahimmian36/hastatus}\end{center}
\begin{samepage}
All five frontier ledgers and the exit-set ledger are distributed; the $10^9$ ledger is also reproduced by either certificate-only program in under a minute, and its SHA-256 (one string, shown on two lines) is
\begin{center}\small
\texttt{b5c113503224dcbcc8aa812d905818f7}\\
\texttt{a6a3bd03e6a9d3fec614f337ec01da6b}
\end{center}
\end{samepage}

\section*{Statements}
\paragraph{Disclosure.} The authors report there are no competing interests to declare.

\paragraph{Funding.} This work was supported by 3kVC and Pareto through Millennium Research.

\paragraph{Declaration of generative AI use.}
Large language models---Anthropic's Claude (Fable 5.1)---were used throughout this work:
to explore and draft arguments, to draft and revise the text, to
write the programs in the artifact, and, in separate sessions
without access to earlier reviews, to review the manuscript for
errors.  The authors directed the work, checked every proof, and
take full responsibility for the content.  Every number quoted in
the paper that arises from our own computations is reproduced by
a named program of the artifact, and the artifact contains that
program's recorded output.

\end{document}

%% file: leanmap.tex
% generated by scripts/lean_index.py --tex; do not edit
\small
\begin{tabular}{@{}>{\raggedright\arraybackslash}p{0.34\textwidth}>{\raggedright\arraybackslash}p{0.62\textwidth}@{}}
\toprule
Statement (numbering of \cite{Li26}) & Lean name(s) \\
\midrule
Lemma 2.1 (equivalence relation) & \texttt{coal\_refl}, \texttt{coal\_symm}, \texttt{coal\_trans} \\
\S 6, unnumbered remark: parity of $\tau$ (classical; one-step form) & \texttt{odd\_card\_divisors\_iff}, \texttt{h\_eq\_of\_parity}, \texttt{h\_ne\_of\_parity\_square} \\
Lemma 4.1 (integer form) & \texttt{card\_divisors\_le\_two\_mul\_sqrt}, \texttt{h\_lt\_sqrt\_add\_two\_sq} \\
Lemma 3.1 / Cor. 3.2 (window form) & \texttt{exists\_iterate\_mem\_window}, \texttt{coal\_of\_window}, \texttt{coal\_below\_of\_cert} \\
(ours) certificate & \texttt{tau\_of\_pairs}, \texttt{certOk} \\
(ours) rungs & \texttt{coal\_below\_1e5}, \texttt{coal\_below\_1e6}, \texttt{coal\_below\_1e7}, \texttt{coal\_below\_1e8} \\
\bottomrule
\end{tabular}

%% file: main.bbl
\begin{thebibliography}{9}
\bibitem{EP414} T.~F. Bloom, Erd\H{o}s Problem \#414, \url{https://www.erdosproblems.com/414}, accessed 2026-09-09.
\bibitem{ErGr80} P.~Erd\H{o}s and R.~L. Graham, \emph{Old and new problems and results in combinatorial number theory}, Monographies de L'Enseignement Math\'ematique 28, Gen\`eve, 1980.
\bibitem{FC} Google DeepMind, \emph{formal-conjectures}: a collection of formalized conjectures in Lean, \url{https://github.com/google-deepmind/formal-conjectures}, file \texttt{FormalConjectures/ErdosProblems/414.lean}, accessed 2026-09-09.
\bibitem{Li26} E.~Li, Square-annular dynamics and coalescence frontiers for $n+\tau(n)$, arXiv:2606.17926, 2026.
\bibitem{OEIS} OEIS Foundation Inc., Entry A064491 in \emph{The On-Line Encyclopedia of Integer Sequences}, \url{https://oeis.org/A064491}, accessed 2026-09-09.
\end{thebibliography}
